\documentclass[a4paper,12pt]{amsart}
\usepackage{amssymb}
\usepackage{cite}
\usepackage{ifthen}
\usepackage[dvips]{graphicx}
\nonstopmode \numberwithin{equation}{section}
\newtheorem*{theoA}{Theorem A}
\newtheorem*{theoB}{Theorem B}
\newtheorem*{theoC}{Theorem C}

\usepackage{amssymb}
\usepackage{ifthen}
\usepackage{graphicx}
\usepackage{amsmath}
\usepackage[T1]{fontenc} 
\usepackage[utf8]{inputenc}
\usepackage[usenames,dvipsnames]{color}
\usepackage{color}
\usepackage[english]{babel}
\usepackage{fancyhdr}
\usepackage{fancybox}
\usepackage{tikz}

\theoremstyle{plain}
\newtheorem{prop}{Proposition}

\newtheorem{ques}{Question}

\newtheorem{conj}{Conjecture}

\theoremstyle{definition}

\newtheorem{cor}{Corollary}[section]
\newtheorem{thm}{Theorem}[section]

\newtheorem{lem}{Lemma}[section]
\newtheorem{prob}{Problem}
\newtheorem{rem}{Remark}[section]

\theoremstyle{plain}

\newtheorem*{lemA}{Lemma A}

\newcounter{minutes}
\divide\time by 60
\newcounter{hours}
\multiply\time by 60
\addtocounter{minutes}{-\time}

\newcounter {own}
\def\theown {\thesection       .\arabic{own}}

\newenvironment{pf}[1][]{%
	\vskip 3mm
	\noindent
	\ifthenelse{\equal{#1}{}}%
	{{\slshape Proof. }}%
	{{\slshape #1.} }%
}%
{\qed\bigskip}

\newcounter{alphabet}

\def\be{\begin{equation}}
	\def\ee{\end{equation}}

\newcommand{\bee}{\begin{enumerate}}
	\newcommand{\eee}{\end{enumerate}}

\newcommand{\blem}{\begin{lem}}
	\newcommand{\elem}{\end{lem}}
\newcommand{\bthm}{\begin{thm}}
	\newcommand{\ethm}{\end{thm}}
\newcommand{\bcor}{\begin{cor}}
	\newcommand{\ecor}{\end{cor}}
\newcommand{\beg}{\begin{examp}}
	\newcommand{\eeg}{\end{examp}}
\newcommand{\begs}{\begin{examples}}
	\newcommand{\eegs}{\end{examples}}

\newcommand{\bdefn}{\begin{defn}}
	\newcommand{\edefn}{\end{defn}}

\newcommand{\bprob}{\begin{prob}}
	\newcommand{\eprob}{\end{prob}}
\newcommand{\bei}{\begin{itemize}}
	\newcommand{\eei}{\end{itemize}}

\newcommand{\bcon}{\begin{conj}}
	\newcommand{\econ}{\end{conj}}
\newcommand{\bcons}{\begin{conjs}}
	\newcommand{\econs}{\end{conjs}}
\newcommand{\bprop}{\begin{prop}}
	\newcommand{\eprop}{\end{prop}}
\newcommand{\br}{\begin{rem}}
	\newcommand{\er}{\end{rem}}
\newcommand{\brs}{\begin{rems}}
	\newcommand{\ers}{\end{rems}}
\newcommand{\bo}{\begin{obser}}
	\newcommand{\eo}{\end{obser}}
\newcommand{\bos}{\begin{obsers}}
	\newcommand{\eos}{\end{obsers}}
\newcommand{\bpf}{\begin{pf}}
	\newcommand{\epf}{\end{pf}}
\newcommand{\ba}{\begin{array}}
	\newcommand{\ea}{\end{array}}
\newcommand{\beq}{\begin{eqnarray}}
	\newcommand{\beqq}{\begin{eqnarray*}}
		\newcommand{\eeq}{\end{eqnarray}}
	\newcommand{\eeqq}{\end{eqnarray*}}

\begin{document}

\title{The Bohr inequality on a simply connected domain and its applications}

\author{Sabir Ahammed}
\address{Sabir Ahammed, Department of Mathematics, Jadavpur University, Kolkata-700032, West Bengal, India.}
\email{sabira.math.rs@jadavpuruniversity.in}

\author{Molla Basir Ahamed$^{*}$}
\address{Molla Basir Ahamed, Department of Mathematics, Jadavpur University, Kolkata-700032, West Bengal, India.}
\email{mbahamed.math@jadavpuruniversity.in}

\author{Partha Pratim Roy}
\address{Partha Pratim Roy, Department of Mathematics, Jadavpur University, Kolkata-700032, West Bengal, India.}
\email{pproy.math.rs@jadavpuruniversity.in}

\subjclass[{AMS} Subject Classification:]{Primary 30A10, 30B10, 30C75, 30H05, 30C35, 40A30 Secondary 30C45}

\def\thefootnote{}
\footnotetext{ {\tiny File:~\jobname.tex,
printed: \number\year-\number\month-\number\day,
          \thehours.\ifnum\theminutes<10{0}\fi\theminutes }
} \makeatletter\def\thefootnote{\@arabic\c@footnote}\makeatother

\begin{abstract} 
In this article, we first establish a generalized Bohr inequality and examine its sharpness for a class of analytic functions $f$ in a simply connected domain $\Omega_\gamma,$ where $0\leq \gamma<1$  with a sequence $\{\varphi_n(r) \}^{\infty}_{n=0}$ of non-negative continuous functions defined on $[0,1)$ such that the series $\sum_{n=0}^{\infty}\varphi_n(r)$ converges locally uniformly on  $[0,1)$. Our results represent twofold generalizations corresponding to those obtained for the classes  $\mathcal{B}(\mathbb{D})$ and $\mathcal{B}(\Omega_{\gamma})$, where \begin{align*}
	\Omega_{\gamma}:=\biggl\{z\in \mathbb{C}: \bigg|z+\dfrac{\gamma}{1-\gamma}\bigg|<\dfrac{1}{1-\gamma}\biggr\}.
\end{align*} As a convolution counterpart, we determine the Bohr radius for hypergeometric function on  $ \Omega_{\gamma} $. Lastly, we establish a generalized Bohr inequality and its sharpness for the class of $ K $-quasiconformal, sense-preserving harmonic maps of the form $f=h+\overline{g}$  in $\Omega_{\gamma}.$
\end{abstract}

\maketitle
\pagestyle{myheadings}
\markboth{ S. Ahammed, M. B. Ahamed and P. P. Roy}{The Bohr inequality for a simply connected domain and its applications}

\section{Introduction}
Let $\mathbb{D}(a;r):=\{z\in \mathbb{C}: |z-a|<r\}$ and let $\mathbb{D}:=\mathbb{D}(0;1)$ be the open unit disk in the complex plane $\mathbb{C}.$ For a given simply connected domain $\Omega$ containing the unit disk $\mathbb{D},$ we denote $\mathcal{H}\left(\Omega\right)$ a class of analytic functions on $\Omega$, and $\mathcal{B}\left(\Omega\right)$ be the class of functions $f\in\mathcal{H}\left(\Omega\right) $ such that $f\left(\Omega\right)\subseteq \mathbb{D}.$ The Bohr radius for the family $\mathcal{B}\left(\Omega\right)$ is defined to be the positive real number $B_{\Omega}\in (0,1)$ given by (see \cite{Fournier-2010}) 
\begin{align*}
B_{\Omega}=\sup \{r\in (0,1): M_f(r)\leq 1\; \mbox{for all}\;f(z)=\sum_{n=0}^{\infty}a_nz^n\in\mathcal{B}\left(\Omega\right), z\in \mathbb{D} \},
\end{align*}
where $M_f(r):=\sum_{n=0}^{\infty}|a_n|r^n$ is the associated majorant series for $f\in \mathcal{B}\left(\Omega\right)$. If $\Omega=\mathbb{D},$ then it is well-known that $B_{\mathbb{D}}=1/3,$ and it is described precisely as follows.
\begin{theoA}(see \cite{Bohr-1914})\label{Th-1.1}
	If $f\in \mathcal{B}\left(\mathbb{D}\right)$, then $M_f(r)\leq 1$ for $ 0\leq r\leq 1/3 $. The number $1/3$ is best possible.
\end{theoA}
The inequality $M_f(r)\leq 1$, which is known as the classical Bohr inequality for $f\in \mathcal{B}\left(\mathbb{D}\right)$,  fails to  hold for any $r>1/3$. In view of this, the constant $ 1/3 $ is best possible. This can be seen by considering the function $f_a(z)=(a-z)/(1-az)$ and by taking $a\in (0,1)$ such that $a$ is sufficiently close to $1$. Theorem A was originally obtained by H. Bohr in \cite{Bohr-1914} for $ 0\leq r\leq 1/6 $. The optimal value $ 1/3 $, which is called the Bohr radius for the disk case, was later established independently by M. Riesz, I. Schur, and F.W. Wiener. Over the past three decades there has been significant interest in Bohr-type inequalities. See \cite{Dixon & BLMS & 1995,Paulsen-PAMS-2004,Blasco-2010,Djakov-JA-20000,Kay & Pon & AASFM & 2019,Abu-Muhanna-survey-2016,Ponnusamy-Shmidt-Starkov-JMAA-2024,Bhowmik-Das-JMAA-2018,Aha-Allu-BMMSS-2022,Ahamed-AMP-2021,Ahamed-CVEE-2021,Hamada- Honda-BMMS-2024,Abu-CVEE-2010,Kay & Pon & Sha & MN & 2018,Huang-Liu-Ponnusamy-MJM-2021,Kayumov-Ponnusamy-JMAA-2018,Liu-Ponnusamy-BMMS-2019,Ponnusamy-Vijayakumar} and the references there.\vspace{1.2mm}
 
The Bohr radius has also been investigated by many researchers in various multidimensional spaces. For example, in $ 1997 $, generalizing a classical one-variable theorem of Bohr, Boas and Khavinson \cite{Boas-Khavinson-PAMS-1997} showed that if an $ n $-variable power series has modulus less than $ 1 $ in the unit polydisc, then the sum of the moduli of the terms is less than $ 1 $ in the polydisc of radius $ 1/(\sqrt{3}n) $. The result of Boas and Khavinson generates later extensive research in multidimensional Bohr inequalities. For different aspects of multidimensional Bohr inequalities, we refer to the articles \cite{Aizn-PAMS-2000,SABIR-MBA-CVEE-2023,Hamada-IJM-2009,Defant-AM-2012,Galicer-TAMS-2020,Liu-Ponnusamy-PAMS-2021,Lin-Liu-Pon-AMS-2023,Kumar-PAMS-2022,Kumar-Manna-JMAA-2023} and references therein. \vspace{1.2mm}

Determination of the Bohr radius has been extended to various classes of function for domains beyond the unit disk. For example, in \cite{Ali- Ng-CVEE-2018}, Ali and Ng investigated Bohr phenomenon in the Poincar\'e disk model of the hyperbolic plane and develops the hyperbolic Bohr inequality for the class of functions mapping $\mathbb{D}^h$ to any concentric subdisk of $\mathbb{D}$, where $\mathbb{D}^h$ is the unit disk in the hyperbolic disk model. As a consequence, the hyperbolic Bohr radius is obtained for analytic self-maps of the hyperbolic unit disk $\mathbb{D}^h$. The hyperbolic Bohr inequality is studied for the class of functions mapping $\mathbb{D}^h$ into a given set lying in an open half-plane of the unit disk. In \cite{Muhanna-Ali-Math.Nachr-2013} Muhanna and Ali have studied Bohr phenomenon to three classes of analytic functions mapping the unit disk respectively into the right half-plane, the slit region, and to the exterior of the unit disk using the properties of the hyperbolic metric. In \cite{Aytuna-Djakov-BLMS-2013}, Bohr property of bases is established in the space of entire functions defined on a compact subset of $\mathbb{C}^n,$ $n\geq 1$. In \cite{Kaptanoglu- Sadık-RJMP-2005}, Kaptanoglu and Sadik have studied Bohr Radii for class of functions defined in a domain bounded by the ellipse. The value of the Bohr radius estimated for elliptical domains of small eccentricity also shows that these domains do not exhibit the Bohr Phenomenon when the eccentricity is large. The Bohr radius was obtained in \cite{Muhana-MN-2017} for analytic functions that map from the unit disk to the punctured unit disk, showing that the radius same as the classical Bohr radius $1/3$. \vspace{1.2mm}

Finding the Bohr radius for a class of functions defined in an arbitrary domain is a difficult task and interesting also. In this paper, our aim is to establish generalized Bohr inequalities for the classes of analytic, and harmonic mappings defined in simply connected domain $\Omega_{\gamma}$ into $\mathbb{C}.$ The paper is organized as follows: In Section 2, we have studied generalized version of Bohr inequality and its sharpness for the class of analytic functions $f\in \mathcal{B}\left(\Omega_{\gamma}\right)$, $0\leq \gamma<1$ with the help of a sequence $\{\varphi_n(r) \}^{\infty}_{n=0}$ of non-negative continuous function in $[0,1)$ such that the series $\sum_{n=0}^{\infty}\varphi_n(r)$ converges locally uniformly on the interval $[0,1)$. By several examples, we show that our results represent twofold generalizations corresponding to those existing results which are obtained for the classes $\mathcal{B}(\mathbb{D})$ and $\mathcal{B}(\Omega_{\gamma})$. In Section 3, convolution counterpart of Bohr radius on simply connected domain $ \Omega_{\gamma} $ is established which generalized the result on unit disk $ \mathbb{D} $. In Section 4, we focus on a class of harmonic mappings, namely, $ K $-quasiconformal sense-preserving harmonic mappings, and derive generalized Bohr inequalities based on existing results. The proof of the main results has been discussed in detail in each section.
\section{Generalization of the Bohr inequality and its sharpness for the class $\mathcal{B}\left(\Omega_{\gamma}\right)$ }
In order to present our results in this section, we need to introduce some basic notations. Following the recent study of the Bohr inequality  (see \cite{Kayumov-Khammatova-Ponnusamy-MJM-2022}), let $\mathcal{F}$ denote the set of all sequence $\varphi=\{\varphi_n(r) \}^{\infty}_{n=0}$ of non-negative continuous function in $[0,1)$ such that the series $\sum_{n=0}^{\infty}\varphi_n(r)$ converges locally uniformly on the interval $[0,1)$. The aim of this section is to discuss not only refined Bohr inequalities and their sharpness for the class $ \mathcal{B}(\Omega_{\gamma}) $, but also to provide some applications of the main results, generalizing several existing findings. To facilitate understanding, we adopt the notations
\begin{align*}
	\Phi_N(r):=\sum_{n=N}^{\infty}\varphi_n(r)\;\;\mbox{and}\;\; 
	\mathcal{A}\left(f_0,\varphi,r\right):=\sum_{n=1}^{\infty}|a_n|^{2n}\left(\dfrac{\varphi_{2n}(r)}{1+|a_0|}+\Phi_{2n+1}(r)\right),
\end{align*}
where $f_0(z)=f_(z)-f(0)$  and $N:=\{1,2,\dots\}.$ In particular, when $\varphi_n(r)=r^n,$ the formulation of $\mathcal{A}\left(f_0,\varphi,r\right)$ simplifies to
\begin{align*}
	\mathcal{A}\left(f_0,r\right):=\left(\dfrac{1}{1+|a_0|}+\dfrac{r}{1-r}\right)\sum_{n=1}^{\infty}|a_n|^{2n}r^{2n}
\end{align*}
This serves as a tool for refining the traditional Bohr inequality.\vspace{1.2mm}

The Bohr inequalities have been investigated for a general class $ \mathcal{F} $ of power series by replacing $ r^n $ by $ \varphi_n(r) $ in the majorant series, where $ \varphi_n(r) $ is a non-negative continuous function on $ [0,1) $ such that $ \sum_{n=0}^{\infty}\varphi_n(r) $ converges locally uniformly with respect to $ r\in [0, 1) $. See the articles \cite{Ponnusamy-Vijayakumar-Wirths-RM-2020,Ponnusamy-Vijayakumar-Wirths-HJM-2021,Kayumov-Khammatova-Ponnusamy-MJM-2022,Chen-Liu-Ponnusamy-RM-2023} and reference therein. For $ 0\leq \gamma<1 $, we consider the simply connected domain  $ \Omega_{\gamma} $ defined by 
 \begin{align*}
 	\Omega_{\gamma}:=\biggl\{z\in \mathbb{C}: \bigg|z+\dfrac{\gamma}{1-\gamma}\bigg|<\dfrac{1}{1-\gamma}\biggr\}.
 \end{align*}
It is clear that $ \Omega_{\gamma} $ contains the unit disk $ \mathbb{D} $, i.e., $\mathbb{D}\subseteq \Omega_{\gamma}$ for $0\leq \gamma<1$. In the following result, sharp coefficient bound for $ f\in\mathcal{B}(\Omega_{\gamma}) $ was obtained and this bound will be instrumental in our study.
\begin{figure}[!htb]
	\begin{center}
		\includegraphics[width=0.55\linewidth]{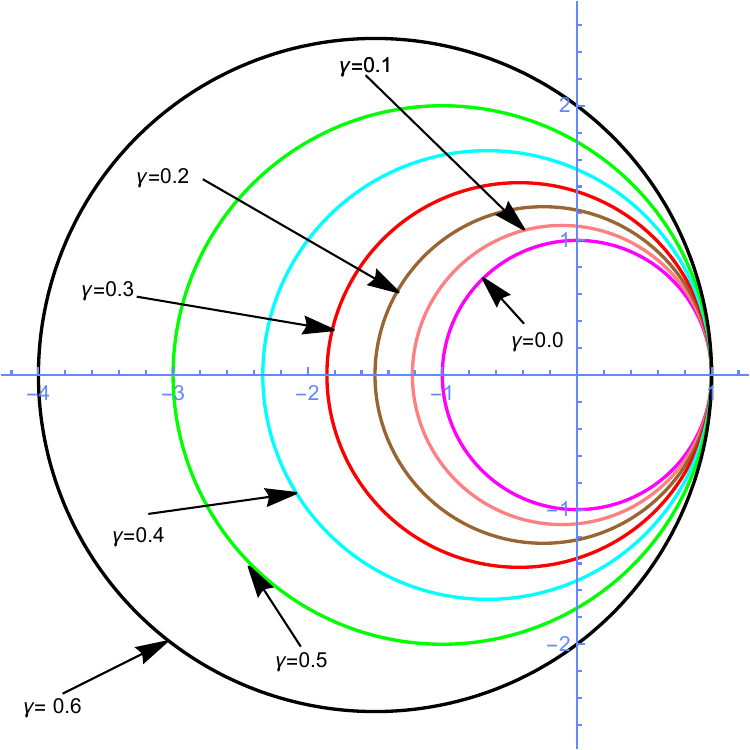}
	\end{center}
	\caption{The graphs of $C_{\gamma}:=\partial \Omega_{\gamma}$ are exhibited for different values of $\gamma=0.0,$ $0.1,$ $0.2,$ $0.3,$ $0.4,$ $0.5$ and $0.6$.}
\end{figure}
\begin{lemA}(see \cite[Lemma 2]{Evdoridis-Ponnusamy-Rasila-RM-2021})
	For $ 0\leq \gamma<1 $, let $ f $ be analytic in $ \Omega_{\gamma} $, bounded by $ 1 $, with power series representation $f(z)=\sum_{n=0}^{\infty}a_nz^n$ on $ \mathbb{D} $. Then 
	\begin{align*}
		|a_n|\leq\dfrac{1-|a_0|^2}{1+\gamma}\;\; \mbox{for}\;\; n\geq 1.
	\end{align*}
\end{lemA}
Since $ \Omega_{\gamma} $ contains the unit disk $ \mathbb{D} $ and it will be convenient to study Bohr-type inequalities for analytic, or harmonic mappings or for other type functions on $ \Omega_{\gamma} $.  However, to made some progress, in \cite{Fournier-2010}, Fournier and Ruscheweyh have extended the classical Bohr inequality for the class $\mathcal{B}(\Omega_{\gamma})$ in the following form.
\begin{theoB}(see \cite[Theorem 1]{Fournier-2010})
	For $ 0\leq\gamma<1 $, let $ f\in\mathcal{B}(\Omega_{\gamma}) $ with $ f(z)=\sum_{n=0}^{\infty}a_nz^n $ in $ \mathbb{D} $. Then 
	\begin{align*}
		B_{f}(r):=\sum_{n=0}^{\infty}|a_n|r^n\leq 1\; \mbox{for}\; |z|=r\leq \rho_{\gamma}:=\frac{1+\gamma}{3+\gamma}.
	\end{align*}
\end{theoB}
It is easy to see that $\rho_0=1/3=B_{\mathbb{D}}$, which is the classical Bohr radius. Improved Bohr inequalities have established on $ \Omega_{\gamma} $ recently in \cite{Ahamed-AASFM-2022,Evdoridis-Ponnusamy-Rasila-RM-2021}. The Bohr inequality has been studied (see \cite{Allu-Halder-PEMS-2023,Allu-Halder-CMB-2023}) for Banach spaces on the domain $ \Omega_{\gamma} $ with help of a sequence $ \{\varphi_n(r)\}_{n=0}^{\infty} \in \mathcal{F}.$ Moreover, in \cite{Kumar-CVEE-2022}, a generalized Bohr inequality for analytic functions $f$ studied on the domain $\Omega_{\gamma}$ with help of a sequence $ \{\varphi_n(r)\}_{n=0}^{\infty}\in \mathcal{F}$. In fact, this result extends the scope of the study in \cite[Theorem 1]{Kayumov-Khammatova-Ponnusamy-MJM-2022}.

\begin{theoC}\label{BS-thm-3.11}(see \cite[Theorem 2.2]{Kumar-CVEE-2022}) Assume that $p\in (0,2].$ Let $f(z)=\sum_{n=0}^{\infty}a_nz^n$ be analytic function with $|f(z)|\leq 1$ in $\Omega_{\gamma}.$ If $\{ \varphi_n(r)\}\in \mathcal{F}$ satisfies the inequality 
	\begin{align}\label{BS-eq-3.6}
		(1+\gamma)\varphi_0(r)>\dfrac{2}{p}\sum_{n=1}^{\infty}\varphi_n(r), 
	\end{align}
	then the following sharp inequality holds:
	\begin{align*}
		\varphi_0(r)|a_0|^p+\sum_{n=1}^{\infty}|a_n|\varphi_n(r)\leq \varphi_0(r)
	\end{align*}
	for $|z|=r\leq R_{\gamma}(p),$ where $R_{\gamma}(p)$ is the minimal positive root in $(0,1)$ of the equation  
	\begin{align}\label{BS-eq-3.2}
		\dfrac{2}{p}\sum_{n=1}^{\infty}\varphi_n(r)=(1+\gamma)\varphi_0(r).
	\end{align}
	In the case when 
	\begin{align*}
		(1+\gamma)\varphi_0(r)<\dfrac{2}{p}\sum_{n=1}^{\infty}\varphi_n(r)
	\end{align*}
	in some interval $( R_{\gamma}(p), R_{\gamma}(p)+\epsilon)$, the number $ R_{\gamma}(p)$ cannot be improved. If the function $\varphi_n(t)$ $(n\geq 0) $ are smooth, then the last condition is equivalent to the inequality 
	\begin{align}\label{BS-eqq-2.3}
		(1+\gamma)\varphi^{\prime}_0(t)<\dfrac{2}{p}\sum_{n=1}^{\infty}\varphi^{\prime}_n(t).
	\end{align}
\end{theoC}
In the study of Bohr phenomenon, one aspect involves the examination of refined Bohr inequalities and proof of their sharpness for certain classes of functions (see \cite{Liu-Ponnusamy-PAMS-2021,SABIR-MBA-CVEE-2023,Liu-Liu-Ponnusamy-BDSM-2021}). While refining a Bohr inequality often yields different Bohr radii for various classes of functions, there remains a curiosity in establishing a refined Bohr inequality for the same class without altering the radius. In this section, we are interested in obtaining a refined version of Theorem C. With this in mind, it's appropriate to raise the following question.
\begin{ques}\label{BS-prob-2.1}
	Can we obtain a sharp refined version of  Theorem C keeping the constant $ R_{\gamma}(p)$ unchanged?
\end{ques}
 We obtain  the following result as a refined version of Theorem C answering the Question \ref{BS-prob-2.1} completely. It is noteworthy that the radius $ R_\gamma(p) $ in Theorem \ref{BS-thm-3.1} is independent of the function $ \Lambda : [0, 1]\to [0, 1] $. In particular, if $ \Lambda\equiv 0 $ for all $ r\in [0, 1] $, then Theorem C can be obtained from Theorem \ref{BS-thm-3.1}.
\begin{thm}\label{BS-thm-3.1}
	Assume that $p\in (0,2]$ and $ \Lambda : [0, 1]\to [0, 1] $ be a function. Let $f(z)=\sum_{n=0}^{\infty}a_nz^n$ be analytic function with $|f(z)|\leq 1$ in $\Omega_{\gamma}.$ If $\{ \varphi_n(r)\}\in \mathcal{F}$ and satisfies \eqref{BS-eq-3.6},
	then the following sharp inequality holds:
	\begin{align*}
	\mathcal{M}^{\Lambda(r)}_f\left(\varphi(r),p,\gamma\right):=\varphi_0(r)|a_0|^p+\sum_{n=1}^{\infty}|a_n|\varphi_n(r)+\Lambda(r)\mathcal{A}\left(f_0,\varphi,r\right)\leq \varphi_0(r)
	\end{align*}
	for $|z|=r\leq R_\gamma(p),$ where $R_\gamma(p)$ is the minimal positive root in $(0,1)$ of the equation \eqref{BS-eq-3.2}.  
	In the case when 
	\begin{align}\label{eeeq-2.4}
		(1+\gamma)\varphi_0(r)<\dfrac{2}{p}\sum_{n=1}^{\infty}\varphi_n(r)
	\end{align}
	in some interval $(R_{\gamma}(p),R_{\gamma}(p)+\epsilon)$, the number $R_{\gamma}(p)$  cannot be improved. If the function $\varphi_n(t)$ $(n\geq 0) $ are smooth, then \eqref{eeeq-2.4} is equivalent to the inequality \eqref{BS-eqq-2.3}.
\end{thm}
As a consequence of Theorem \ref{BS-thm-3.1}, in particular, if  $\gamma=0$ and $\Lambda(r)=1$ for all $ r\in [0, 1] $, then we get the following corollary.
\begin{cor}(see \cite[Theorem 1]{Ponnusamy-Vijayakumar-Wirths-HJM-2021})
	Assume that $p\in (0,2]$ and let $f(z)=\sum_{n=0}^{\infty}a_nz^n$ be analytic function with $|f(z)|\leq 1$ in $\mathbb{D}.$ If $\{ \varphi_n(r)\}\in \mathcal{F}$ satisfies \eqref{BS-eq-3.6} with $\gamma=0$,
then the following sharp inequality holds:
\begin{align*}
	\varphi_0(r)|a_0|^p+\sum_{n=1}^{\infty}|a_n|\varphi_n(r)+\mathcal{A}\left(f_0,\varphi,r\right)\leq \varphi_0(r)
\end{align*}
for $|z|=r\leq R_0(p),$ where $R_0(p)$ is the minimal positive root in $(0,1)$ of the equation $ 	\varphi_0(r)=\dfrac{2}{p}\sum_{n=1}^{\infty}\varphi_n(r). $  In the case when
\begin{align}\label{BSP-eq-2.5}
	\varphi_0(r)<\dfrac{2}{p}\sum_{n=1}^{\infty}\varphi_n(r)
\end{align}
in some interval $(R_{0}(p),R_{0}(p)+\epsilon)$, the number $R_{0}(p)$  cannot be improved. If the function $\varphi_n(t)$ $(n\geq 0) $ are smooth, then the condition \eqref{BSP-eq-2.5} is equivalent to the inequality \begin{align*}
	\varphi^{\prime}_0(t)<\dfrac{2}{p}\sum_{n=1}^{\infty}\varphi^{\prime}_n(t).
\end{align*}
\end{cor}
\subsection{Application of Theorem \ref{BS-thm-3.1}}
As application of Theorem \ref{BS-thm-3.1}, the following results are the counterparts of refined version of Bohr's theorem for the class $\mathcal{B}(\Omega_{\gamma})$. The inequalities in these results are generalization of the Bohr inequalities for the class $\mathcal{B}(\mathbb{D})$.
\begin{itemize}
	\item [1.] Let $\varphi_n(r)=r^n$, and $p=1, 2$, $ \Lambda(r)=1 $ for all $ r\in [0, 1] $ in  Theorem \ref{BS-thm-3.1}.  We obtain the following inequalities which are refined version of \cite[Theorem 2]{Evdoridis-Ponnusamy-Rasila-RM-2021} for the class $\mathcal{B}(\Omega_{\gamma})$
	\begin{align*}
		\sum_{n=0}^{\infty}|a_n|r^n+	\mathcal{A}\left(f_0,r\right)\leq 1\;\;\mbox{ for}\;\; r\leq R_{\gamma}(1)= \dfrac{1+\gamma}{3+\gamma}.
	\end{align*}
	Moreover,
	\begin{align}\label{Eq-2.5}
		|a_0|^2+\sum_{n=1}^{\infty}|a_n|r^n+	\mathcal{A}\left(f_0,r\right)\leq 1\;\;\mbox{ for}\;\; r\leq R_{\gamma}(2)=\dfrac{1+\gamma}{2+\gamma}.
	\end{align}
	Both the radii $ R_{\gamma}(1) $ and $ R_{\gamma}(2) $ are best possible. In particular, for $\gamma=0$, the inequality (3) of \cite[Theorem  2]{Ponnusamy-Vijayakumar-Wirths-RM-2020} coincides with \eqref{Eq-2.5}. \vspace{2mm}
	
	\item [2.] Let $\varphi_{2n}(r)=r^{2n}$ and  $\varphi_{2n+1}(r)=0,$ $n\geq 0$, $ \Lambda(r)=1 $ for all $ r\in [0, 1] $ in  Theorem \ref{BS-thm-3.1}. We see that 
	\begin{align*}
		\mathcal{A}\left(f_0,\varphi,r\right)=\frac{1}{1+|a_0|}\sum_{n=1}^{\infty}|a_n|^{2n}r^{2n}
	\end{align*}
	and the refined inequality 
	\begin{align}\label{Eq-2.6}
		|a_0|^p+\sum_{n=1}^{\infty}|a_{2n}|r^{2n}+\frac{1}{1+|a_0|}\sum_{n=1}^{\infty}|a_n|^{2n}r^{2n}\leq 1
	\end{align}
	holds for $r\leq R^{\gamma}_2(p)$, where $R^{\gamma}_2(p)=\sqrt{\frac{p(1+\gamma)}{2+p(1+\gamma)}}$. The radius $R^{\gamma}_2(p)<1$ is best possible. Note that \eqref{Eq-2.6} is a refined version of the inequality $|a_0|^p+\sum_{n=1}^{\infty}|a_{2n}|r^{2n}\leq 1$ in \cite[p. 967]{Kumar-CVEE-2022} keeping the radius $ R^{\gamma}_2(p) $ unchanged. \vspace{2mm}
	\item [3.] Let  $\varphi_0(r)=1$,  $\varphi_{2n}(r)=0$ and  $\varphi_{2n-1}(r)=r^{2n-1}, $ $n\geq 0$ in  Theorem \ref{BS-thm-3.1}. Then we have
	\begin{align*}
	\mathcal{A}\left(f_0,\varphi,r\right)=\sum_{n=1}^{\infty}|a_{2n-1}|^{4n-2}\frac{r^{4n+1}}{1-r^2}
	\end{align*}
	and we see that
	\begin{align*}
		|a_0|^p+\sum_{n=1}^{\infty}|a_{2n-1}|r^{2n-1}+\Lambda\sum_{n=1}^{\infty}|a_{2n-1}|^{4n-2}\frac{r^{4n+1}}{1-r^2}\leq 1 \;\;\mbox{for}\;\;r\leq R^{\gamma}_3(p),
	\end{align*}
	where $R^{\gamma}_3(p)=\frac{\sqrt{1+p^2(1+\gamma)}-1}{p(1+\gamma)}$. The radius $R^{\gamma}_3(p)(<1)$ (independent of $\Lambda$) is best possible.\vspace{2mm}
	\item [4.] Let  $\varphi_0(r)=1$, $ \Lambda=1 $ for all $ r\in [0, 1] $, $ \varphi_{n}(r)=0$\; for $1\leq n<N$ and $\varphi_{n}(r)=(n+1)r^n,$  $n\geq N$, 
	\begin{align*}
		C^{|a_0|}_n(r)=\frac{(2n+1)r^{2n}}{1+|a_0|}+\frac{r^{2n+1}(2+2n(1-r)-r)}{(1-r)^2}
	\end{align*} in  Theorem \ref{BS-thm-3.1}. Then the refined inequality
	\begin{align*}
	|a_0|^p+\sum_{n=N}^{\infty}(n+1)|a_n|r^n+\Lambda(r)\sum_{n=N}^{\infty}|a_n|^{2n}C^{|a_0|}_n(r)\leq 1,
	\end{align*}
	holds for $r\leq R^{\gamma,4}_{p,N}$, where $R^{\gamma, 4}_{p,N}$ (independent of $\Lambda(r)$) is the smallest positive root in $(0,1)$ of the equation $2r^N(1+N-Nr)-p(1+\gamma)(1-r)^2=0$. The radius $R^{\gamma,4}_{p,N}$ is best possible. In  particular, if $N=1$, then we see that $R^{\gamma,4}_{p,1}=1-\sqrt{\frac{2}{p(1+\gamma)+2}}$ and $R^{\gamma,4}_{p,1}$ is best possible. Further, in particular, if $\gamma=0$, then we see that $R^{0,4}_{p,1}=1-\sqrt{\frac{2}{p+2}}=R_2(p)$ in \cite[p. 4]{Kayumov-Khammatova-Ponnusamy-MJM-2022}. Thus our result are two fold generalization corresponding to the results for the class $\mathcal{B}(\mathbb{D})$ and $\mathcal{B}(\Omega_{\gamma})$. \vspace{2mm}
		
		\item [5.] The Lerch trancendent $ 	\Phi(z,s,a) $ is generalization of the Hurwitz-zeta function and polylogarithm function. It is classically defined by 
		\begin{align*}
			\Phi(z,s,a)=\sum_{n=0}^{\infty}\frac{z^n}{(a+n)^s},\;\;\mbox{for}\;|z|<1
		\end{align*}
		and $a\neq0,-1\dots.$ \vspace{1.5mm}
		
		In special case it gives the Poly logarithmic function, \textit{i.e.}
		\begin{align*}
			\Phi(z,n,1)=\frac{Li_n(z)}{z},
		\end{align*}
		where $Li_n(z)$ is the Poly logarithmic function.\vspace{1.2mm}
		 
		Let  $\varphi_0(r)=1, \varphi_{n}(r)=0$\; for $1\leq n<N$ and $\varphi_{n}(r)=n^{\alpha}r^n,$  $n\geq N$ in  Theorem \ref{BS-thm-3.1}. Then refined inequality
			\begin{align*}
				|a_0|^p+\sum_{n=N}^{\infty}|a_n|n^{\alpha}r^n+\Lambda(r)\sum_{n=N}^{\infty}|a_n|^{2n}\bigg(\frac{(2n)^{\alpha}r^{2n}}{1+|a_0|}+\Phi(r,\alpha,2n+1)\bigg)\leq 1 
			\end{align*}
			for $r\leq R^{\alpha,\gamma,5}_{p,N},$
			where $R^{\alpha,\gamma,5}_{p,N}$ is the smallest positive root in $ (0, 1) $ of the equation $p(1+\gamma)-2\Phi(r,-\alpha,N)r^n=0$. The radius $R^{\alpha,\gamma,5}_{p,N}$ is best possible.\vspace{1.2mm} 
			
			In particular cases, the Bohr radius can be determined explicitly and this is shown below.
			\begin{enumerate}
				\item[(i)] For $\alpha=1$, and 
				\begin{align*}
					A^{|a_0|}_n(r):=\frac{2nr^{2n}}{1+|a_0|}+\frac{(1+2n(1-r))r^{2n+1}}{(1-r)^2}.
				\end{align*} we see that the inequality 
				\begin{align*}
					|a_0|^p+\sum_{n=N}^{\infty}|a_n|nr^n+\Lambda(r)\sum_{n=N}^{\infty}|a_n|^{2n}A^{|a_0|}_n(r)\leq 1
				\end{align*}
				holds for $r\leq R^{1,\gamma,5}_{p,N},$ where $R^{1,\gamma,5}_{p,N}$ is the smallest positive root in $ (0, 1) $ of the equation $2r^N(N(1-r)+r)-p(1+\gamma)(1-r)^2=0$. The radius $R^{1,\gamma,5}_{p,N}$ is best possible. 	
				When $N=1$, we see that 
		\begin{align*}
		R^{1,\gamma,5}_{p,1}=\frac{2(p(1+\gamma)+1)-\sqrt{(2(p(1+\gamma)+1))^2-4(1+\gamma)^2p^2}}{2p(1+\gamma)}
		\end{align*} 
		and $R^{1,\gamma,5}_{p,1}$ is best possible. It is worth to mention that for the class $\mathcal{B}(\mathbb{D})$, we have $R^{1,0,5}_{p,1}=\frac{p+1-\sqrt{2p+1}}{p}=R_3(p)$ (see \cite[p. 4]{Kayumov-Khammatova-Ponnusamy-MJM-2022}).\vspace{2mm}
				
				\item[(ii)] For $\alpha=2$, and 
				\begin{align*}
					B^{|a_0|}_n(r):=\frac{4n^2r^{2n}}{1+|a_0|}+\frac{(1+4n(1-r)+4n^2(1-r)^2+r)r^{1+2n}}{(1-r)^3},
				\end{align*} the refined Bohr inequality
				\begin{align*}
					|a_0|^p+\sum_{n=N}^{\infty}|a_n|n^2r^n+\Lambda(r)\sum_{n=N}^{\infty}|a_n|^{2n}B^{|a_0|}_n(r)\leq 1
				\end{align*}
				holds for $r\leq R^{2,\gamma,5}_{p,N},$ where $R^{2,\gamma,5}_{p,N}$ is the smallest positive root in $ (0, 1) $ of the equation $2r^N(N(1-r)+r)-p(1+\gamma)(1-r)^2=0$. The radius $R^{2,\gamma,5}_{p,N}$ is best possible. Moreover, when $N=1$, we see that 
				\begin{align*}
					R^{2,\gamma,5}_{p,1}=\frac{2(p(1+\gamma)+1)-\sqrt{(2(p(1+\gamma)+1))^2-4(1+\gamma)^2p^2}}{2p(1+\gamma)}
				\end{align*} 
				and $R^{2,\gamma,5}_{p,1}$ is best possible. It is worth mentioning that for the class $\mathcal{B}(\mathbb{D})$, we have $	R^{2,0,5}_{p,1}=R_4(p)$, where $R^{2,0,5}_{p,1}$ is the minimal positive root in $ (0, 1) $ of the equation $p(1-r)^3-2r(1+r)=0$ (see \cite[p. 5]{Kayumov-Khammatova-Ponnusamy-MJM-2022}).
			\end{enumerate}
\end{itemize}
\begin{proof}[\bf{Proof of Theorem \ref{BS-thm-3.1}}]
	Let $a=|a_0|<1.$ Since $f(z)=\sum_{n=0}^{\infty}a_nz^n$ is an analytic function with $|f(z)|\leq 1$ in $\Omega_{\gamma}$,   using the Lemma A, we obtain
	\begin{align*}
			&\mathcal{M}^{\Lambda(r)}_f\left(\varphi(r),p,\gamma\right)\\&\leq a^p\varphi_0(r)+\dfrac{(1-a^2)}{1+\gamma}\sum_{n=1}^{\infty}\varphi_n(r)+\Lambda(r)\left(\dfrac{1-a^2}{1+\gamma}\right)^2\sum_{n=1}^{\infty}\left(\dfrac{\varphi_{2n}(r)}{1+a}+\Phi_{2n+1}(r)\right)\\&=\varphi_0(r)+\dfrac{(1-a^2)}{1+\gamma}\bigg(\sum_{n=1}^{\infty}\varphi_n(r)+\Lambda(r)\left(\dfrac{1-a^2}{1+\gamma}\right)\sum_{n=1}^{\infty}\left(\dfrac{\varphi_{2n}(r)}{1+a}+\Phi_{2n+1}(r)\right)\\&\quad-\left(\dfrac{1-a^p}{1-a^2}\right)(1+\gamma)\varphi_0(r)\bigg).
	\end{align*} 
	To proceed further in the proof, we utilize the following inequality as provided in \cite{Kayumov-Khammatova-Ponnusamy-MJM-2022}
	\begin{align}\label{eee-2.5}
		\dfrac{1-t^p}{1-t^2}\geq \dfrac{p}{2}\;\;\;\mbox{for\; all}\;\; t\in [0,1)\;\;\mbox{and}\;\; p\in(0,2].
	\end{align}
By \eqref{eee-2.5}, it becomes clear that
	\begin{align*}
			&\mathcal{M}^{\Lambda(r)}_f\left(\varphi(r),p,\gamma\right)\\&\leq \varphi_0(r)+\dfrac{(1-a^2)}{1+\gamma}\bigg(\sum_{n=1}^{\infty}\varphi_n(r)+\Lambda(r)\left(\dfrac{1-a^2}{1+\gamma}\right)\sum_{n=1}^{\infty}\left(\dfrac{\varphi_{2n}(r)}{1+a}+\Phi_{2n+1}(r)\right)\\&\quad-\left(\frac{p}{2}\right)(1+\gamma)\varphi_0(r)\bigg)\\& =\varphi_0(r)+\dfrac{(1-a^2)}{1+\gamma}H^{\Lambda(r)}_{\gamma}(a,p),
	\end{align*}
	where
	\begin{align*}
		H^{\Lambda(r)}_{\gamma}(a,p):=\sum_{n=1}^{\infty}\varphi_n(r)+\Lambda(r)\left(\dfrac{1-a^2}{1+\gamma}\right)\sum_{n=1}^{\infty}\left(\dfrac{\varphi_{2n}(r)}{1+a}+\Phi_{2n+1}(r)\right)-\dfrac{p(1+\gamma)}{2}\varphi_0(r).
	\end{align*}
In order to obtain the desire inequality of the result, it is sufficient to show that $H^{\Lambda(r)}_{\gamma}(a,p)\leq 0$ for $|z|=r\leq R_\gamma(p).$ Choosing $a$ very close to $1$ \textit{i.e.} taking $a\rightarrow 1^{-}$ and using the inequality \eqref{BS-eq-3.6}, we obtain 
	\begin{align*}
		\lim\limits_{a\rightarrow 1^-}H^{\Lambda(r)}_{\gamma}(a,p)=\sum_{n=1}^{\infty}\varphi_n(r)-\dfrac{	p(1+\gamma)}{2}\varphi_0(r)\leq 0.
	\end{align*}
	Thus, the desired inequality   $\mathcal{M}^{\Lambda(r)}_f\left(\varphi(r), p,\gamma\right)\leq \varphi_0(r)$ is obtained for $r\leq R_\gamma(p).$\vspace{1.1mm}
	
	Next part of the proof is to show that the radius $R_\gamma(p)$ is best possible. Henceforth, for $ a\in (0,1) $, we consider function 
	\begin{align}\label{eeee-2.6}
		h_a(z)=\dfrac{a-\gamma-(1-\gamma)z}{1-a\gamma-a(1-\gamma)}\;\;\mbox{for}\;\; z\in \mathbb{D}.
	\end{align} 
	We define two functions $ H_1 $ and $ H_2 $ by $H_1:\mathbb{D}\rightarrow\mathbb{D}$ by $H_1(z)=(a-z)/(1-az)$ and $H_2:\Omega_{\gamma}\rightarrow\mathbb{D}$ by $H_2(z)=(1-\gamma)z+\gamma.$ Then, it is easy to see that the function $h_a=H_1\circ H_2$ maps univalently onto $\mathbb{D}.$ We note that $h_a$ is analytic in $\Omega_{\gamma}$ and $|h_a|\leq1.$ We see that $ h_a $ has the following power series representation
	\begin{align*}
		h_a(z)=\dfrac{a-\gamma-(1-\gamma)z}{1-a\gamma-a(1-\gamma)}=a_0-\sum_{n=1}^{\infty}a_nz^n\;\;\mbox{for}\;\;z\in\mathbb{D},
	\end{align*}   
	where
	\begin{align}\label{eee-2.7}
		a_0=\left(\dfrac{a-\gamma}{1-a\gamma}\right)\;\mbox{and}\; a_n=\left(\dfrac{1-a^2}{a(1-a\gamma)}\left(\dfrac{a(1-\gamma)}{1-a\gamma}\right)^n\right)\;\mbox{for}\; n\in \mathbb{N}.
	\end{align}
	In view of the power series expansion of $h_a$ with the coefficients $ a_n $ as given in \eqref{eee-2.7}, a routine computation gives that
	\begin{align*}
		&\mathcal{M}^{\Lambda(r)}_{h_a}\left(\varphi(r), p, \gamma\right)\\&=\left(\dfrac{a-\gamma}{1-a\gamma}\right)^p\varphi_0(r)+\sum_{n=1}^{\infty}\dfrac{1-a^2}{a(1-a\gamma)}\left(\dfrac{a(1-\gamma)}{1-a\gamma}\right)^n\varphi_n(r)\\&\quad+\Lambda(r)\sum_{n=1}^{\infty}\left(\dfrac{1-a^2}{a(1-a\gamma)}\right)^2\left(\dfrac{a(1-\gamma)}{1-a\gamma}\right)^{2n}\left(\dfrac{\varphi_{2n}(r)}{1+\left(\dfrac{a-\gamma}{1-a\gamma}\right)}+\Phi_{2n+1}(r)\right)\\&=\varphi_0(r)+\dfrac{(1-a)}{1+\gamma}\left(2\sum_{n=1}^{\infty}\varphi_n(r)-p(1+\gamma)\varphi_0(r)\right)\\&\quad+(1-a)\left(\sum_{n=1}^{\infty}\dfrac{(1+a)}{a(1-a\gamma)}\left(\dfrac{a(1-\gamma)}{1-a\gamma}\right)^n\varphi_n(r)-\dfrac{2}{1-\gamma}\sum_{n=1}^{\infty}\varphi_n(r)\right)\\&\quad+(1-a)\left(p\left(\dfrac{1+\gamma}{1-\gamma}\right)+\frac{1}{1-a}\bigg(\left(\dfrac{a-\gamma}{1-a\gamma}\right)^p-1\bigg)\right)\varphi_0(r)\\&\quad+\Lambda(r)\left(\dfrac{1-a^2}{a(1-a\gamma)}\right)^2\sum_{n=1}^{\infty}\left(\dfrac{a(1-\gamma)}{a(1-a\gamma)}\right)^{2n}\left(\dfrac{\varphi_{2n}(r)}{1+\dfrac{a-\gamma}{1-a\gamma}}+\Phi_{2n+1}(r)\right).
	\end{align*}
Next, it is a simple task to verify that
	\[\lim\limits_{a\to 1^-}\left(\sum_{n=1}^{\infty}\dfrac{(1+a)}{a(1-a\gamma)}\left(\dfrac{a(1-\gamma)}{1-a\gamma}\right)^n\varphi_n(r)-\dfrac{2}{1-\gamma}\sum_{n=1}^{\infty}\varphi_n(r)\right)=0,\]
\[\lim\limits_{a\to 1^-}\left(p\left(\dfrac{1+\gamma}{1-\gamma}\right)+\frac{1}{1-a}\bigg(\left(\dfrac{a-\gamma}{1-a\gamma}\right)^p-1\bigg)\right)=0,\]
\[\lim\limits_{a\to 1^-}\left(\dfrac{1-a^2}{a(1-a\gamma)}\right)^2\sum_{n=1}^{\infty}\left(\dfrac{a(1-\gamma)}{a(1-a\gamma)}\right)^{2n}\left(\dfrac{\varphi_{2n}(r)}{1+\dfrac{a-\gamma}{1-a\gamma}}+\Phi_{2n+1}(r)\right)=0.\]
	Thus, it follows that 
	\begin{align*}
	\mathcal{M}^{\Lambda(r)}_{h_a}\left(\varphi(r), p, \gamma\right)=\varphi_0(r)+\dfrac{(1-a)}{1+\gamma}\left(2\sum_{n=1}^{\infty}\varphi_n(r)-p(1+\gamma)\varphi_0(r)\right)+O\left((1-a)^2\right)
	\end{align*}
as $ a $ tends to $ 1^- $. Next, in view of \eqref{eeeq-2.4}, it is easy to see that 
	\[\mathcal{M}^{\Lambda(r)}_{h_a}\left(\varphi(r), p, \gamma\right)>\varphi_0(r)\]
	 when $a$ is very close to $1$ \textit{i.e.} $a\rightarrow 1^{-}$ and $r\in\left(R_\gamma(p),R_\gamma(p)+\epsilon \right)$ which shows that the radius $R_\gamma(p)$ cannot be improved further. This completes the proof.
	\end{proof}
\section{Convolution Counterpart of Bohr Radius on simply connected domain}
For two analytic functions $f$ and $g$ with power series representations $f(z)=\sum_{n=0}^{\infty}a_nz^n$ and $g(z)=\sum_{n=0}^{\infty}$ in $\mathbb{D}$, we define the Hadamard product (or convolution) $f*g$ of $f$ and $g$ by the power series 
    \begin{align*}
 	(f*g)(z)=\sum_{n=0}^{\infty}a_nb_nz^n,\;\;z\in\mathbb{D}.
 	 \end{align*}
 	It is clear that $ * $ satisfies the commutative property $f*g=g*f$.	For the Gaussian hypergeometric function $F(z):={}_2 F_1(a,b;c;z)$ defined by the power series expansion 
 	\begin{align*}
 		F(z)=\sum_{n=0}^{\infty}\gamma_nz^n=\frac{(a)_n(b)_n}{(c)_n}\frac{z^n}{n!}=1+\frac{ab}{c}\frac{z}{1!}+\frac{a(a+1)b(b+1)}{c(c+1)}\frac{z^2}{2!}+\ldots.
 	\end{align*} 
 It is undefined (or infinite) if $ c $ equals a non-positive integer. Here $ (x)_n $ is the (rising) Pochhammer symbol, which is defined by:
 \begin{align*}
 	(x)_n:=\begin{cases}
 		1, \;\;\;\;\;\;\;\;\;\;\;\;\;\;\;\;\;\;\;\;\;\;\;\;\;\;\;\;\;\;\;\;\;\;\; \mbox{if}\; n=0,\vspace{2mm}\\
 		x(x+1)\cdots (x+n-1),\; \mbox{if}\; n>0.
 	\end{cases}
 \end{align*}
 	We consider here the convolution operator of the form 
 	\begin{align*}
 		(F*f)(z)=\sum_{n=0}^{\infty}\gamma_na_nz^n.
 	\end{align*}
 	The Gaussian or ordinary hypergeometric function $ {}_2 F_1(a,b;c;z) $ is a special function represented by a hypergeometric series, encompassing numerous other special functions as specific or limiting cases. It serves as a solution to a second-order linear ordinary differential equation (ODE). Any second-order linear ODE with three regular singular points can be transformed into this equation. Riemann showed that when examining the second-order differential equation for $ {}_2 F_1(a,b;c;z) $ in the complex plane, its characterization (on the Riemann sphere) is defined by its three regular singularities. The differentials of Gaussian hypergeometric function defined by 
 	\begin{align*}
 		\frac{d^n}{dz^n}{}_2 F_1(a,b;c;z)=\frac{(a)_n(b)_n}{(c)_n}{}_2 F_1(a+n,b+n;c+n;z).
 	\end{align*}
 	It is easy to see that the function ${}_2 F_1(a,b;c;z)$ is analytic in $\mathbb{D}$ and in particular, ${}_2 F_1(a,1;1;z)=(1-z)^{-a}$. In the exceptional case $c=-m, m=0,1,2,....,{}_2 F_1(a,b;c;z)$ is defined if $a=-j$ or $b=-j$, where $j=0,1,2...,$ and $j\leq m$. Clearly, if $a=-m$, a negative integer, then $F(a,b;c;z)$ becomes a polynomial of degree $m$ in $z$.\vspace{1.2mm}
 	
 	As an application of Theorem \ref{BS-thm-3.1}, we derive the following result for the Gaussian hypergeometric function, precisely determining the Bohr radius.
 	\begin{thm}\label{Th-3.1}
 	Let $p\in (0,2]$ and  $f(z)=\sum_{n=0}^{\infty}a_nz^n$ be analytic function with $|f(z)|\leq 1$ in $\Omega_{\gamma}.$ Assume that $a,b,c>-1$ such that all $\gamma_n$ have the same sign for $n\geq1$. If $\{ \varphi_n(r)\}\in \mathcal{F}$,
 	then the following sharp inequality holds:
 	\begin{align*}
 	\mathcal{M}^{0}_f\left(\varphi,p,\gamma\right):=|a_0|^p+\sum_{n=1}^{\infty}|\gamma_n||a_n|r^n\leq 1\;\;\mbox{for all}\;r\leq R^{\gamma}_p,
 	\end{align*}
 	where $R^{\gamma}_p$ is the minimal positive root of the equation $|F(a,b;c;x)-1|=(1+\gamma){p}/{2}$. The number $R^{\gamma}_p$ cannot be improved.
 	\end{thm}
 	In particular when $\Omega_{\gamma}=\mathbb{D}$, from Theorem \ref{Th-3.1}, we obtain the following corollary which is exactly a result in \cite{Kayumov-Khammatova-Ponnusamy-MJM-2022}.
 	\begin{cor}(see \cite[Theorem 2]{Kayumov-Khammatova-Ponnusamy-MJM-2022})
 	 Let $f(z)=\sum_{n=0}^{\infty}a_nz^n$ be analytic function belongs to $\mathcal{B}$ and $p\in(0,2]$. Assume that $a,b,c>-1$ such that all $\gamma_n$ have the same sign for $n\geq1$.
 		Then the following sharp inequality holds:
 		\begin{align*}
 			|a_0|^p+\sum_{n=1}^{\infty}|\gamma_n||a_n|r^n\leq 1\;\;\mbox{for all}\;r\leq R_p^{0}=R_p,
 		\end{align*}
 		where $R_p$ is the minimal positive root of the equation $|F(a,b;c;x)-1|={p}/{2}$. The number $R_p$ cannot be improved.
 	\end{cor}
 Further, we have the following remark on Theorem \ref{Th-3.1}.
 	\begin{rem}
 		\begin{enumerate}
 			\item[(i)] 	The Bohr radius can, in certain cases, be explicitly determined. For example, the cases where $b=1=c$ or $a=1=c$ or $a=1=b$, in such cases, we see that $F(z)=(1-z)^{-a}$ or $F(z)=(1-z)^{-b}$ or $F(z)=(1-z)^{-c}$, respectively. Then for all such $F(z)$, we have
 			\begin{align}\label{Eqq-3.1}
 				|F(a,b;c;r)-1|=(1+\gamma)\frac{p}{2}\;\;\mbox{i.e.,}\;R^\gamma_p=1-\bigg(\frac{2}{2+(1+\gamma)p}\bigg)^{\frac{1}{y}},
 			\end{align}
 			where $y=a,b,c$ respectively.\vspace{2mm}
 			\item[(ii)] In particular, for $y=1$ in \eqref{Eqq-3.1}, we observe that
 			\begin{align*}
 				R^\gamma_p=1-\frac{2}{2+(1+\gamma)p}=\frac{(1+\gamma)p}{2+(1+\gamma)p}.
 			\end{align*}
 			Further, when $\gamma=0$, we observe that the radius is  $R^\gamma_p=p/(1+p)=R$ (see \cite[p.6]{Kayumov-Khammatova-Ponnusamy-MJM-2022}). Thus, our result extends those of Kayumov \emph{et al.} \cite{Kayumov-Khammatova-Ponnusamy-MJM-2022} in a large extent.
 		\end{enumerate}
 	\end{rem}
 	\begin{proof}[\bf Proof of Theorem \ref{Th-3.1}]
 	    To prove the result, we apply Theorem \ref{BS-thm-3.1} with $\Lambda$=0. Set $\varphi_n(r)=|\gamma_n|r^n$ and remark that $\gamma_0=1.$ Let us also note that all $\gamma_n$ have the same sign \textit{i.e.,} either $\gamma_n>0$ for all $n$ or $\gamma_n<0$ for all $n$. Thus, we see that 
 	   \begin{align*}
 	   	|F(a,b;c;r)-1|=\sum_{n=1}^{\infty}\varphi_n(r).
 	   \end{align*}
 	   By utilizing Lemma A, we obtain
 	   \begin{align*}
 	   	\mathcal{M}^{0}_f\left(\varphi,p,\gamma\right)\leq&|a_0|^p+\frac{(1-|a_0|^2)}{1+\gamma}\sum_{n=1}^{\infty}|\gamma_n|r^n\\&=1+\frac{(1-|a_0|^2)}{1+\gamma}\bigg(\sum_{n=1}^{\infty}|\gamma_n|r^n-\frac{(1-|a_0|^p)}{(1-|a_0|^2)}(1+\gamma)\bigg).
 	    \end{align*}
 	    Taking into account \eqref{eee-2.5}, we observe that
 	    \begin{align*}
 	    \mathcal{M}^{0}_f\left(\varphi,p,\gamma\right)\leq&1+\frac{(1-|a_0|^2)}{1+\gamma}\bigg(|F(a,b;c;r)-1|-\frac{p}{2}(1+\gamma)\bigg).
 	    \end{align*}
 	    Then we have $\mathcal{M}^{0}_f\left(\varphi,p,\gamma\right)\leq1$ for $r\leq R^\gamma_p$, where $R^\gamma_p$ is the minimal root of the equation
 	    \begin{align*}
 	    	|F(a,b;c;x)-1|=(1+\gamma)\frac{p}{2}.
 	    \end{align*}
 	 Considering the function defined by \eqref{eeee-2.6} with the choice of $\varphi_n(r)=|\gamma_n|r^n$, it can be shown that the radius $R^\gamma_p$ is best possible. This completes the proof.
 	\end{proof}
\section{Generalized Bohr inequalities for the class of $K$-quasiconformal harmonic mappings on $\Omega_{\gamma}$}
 A harmonic mapping in the unit disk $ \mathbb{D} $ is a complex-valued function $ f=u+iv $ of $ z=x+iy $ in $ \mathbb{D} $, which satisfies the Laplace equation $ \Delta f=4f_{z\bar{z}}=0 $, where $ f_{z}=(f_{x}-if_{y})/2 $ and $ f_{\bar{z}}=(f_{x}+if_{y})/2 $ and $ u $ and $ v $ are real-valued harmonic functions in $ \mathbb{D} $. It follows that the function $ f $ admits the canonical representation $ f=h+\bar{g} $, where $ h $ and $ g $ are analytic in $ \mathbb{D} $. The Jacobian  matrix of $ f $ is denoted by 
\begin{equation*}
	J_f(z) = 
	\begin{pmatrix}
		u_x & u_y\\
		v_x & v_y
	\end{pmatrix}.
\end{equation*}
Therefore, the Jacobian of $ f $ is 
$ |J_f(z)| $ of $ f=h+\overline{g} $ is given by $ |J_f(z)|=|h^{\prime}|^2-|g^{\prime}|^2 $. We say that $ f $ is sense-preserving in $ \mathbb{D} $ if $ |J_f(z)|>0 $ in $ \mathbb{D} $. Consequently, $ f $ is locally univalent and sense-preserving in $ \mathbb{D} $ if, and only if, $ |J_f(z)|>0 $ in $ \mathbb{D} $; or equivalently if $ h^{\prime}\neq 0 $ in $ \mathbb{D} $ and the dilation $ \omega_f := \omega=g^{\prime}/h^{\prime} $ has the property that $ |\omega(z)|<1 $ in $ \mathbb{D} $.\vspace{1.2mm}

Harmonic mappings play the natural role in parameterizing minimal surfaces in the context of differential geometry. Planar harmonic mappings have application not only in the differential geometry but also in the various fields of engineering, physics, operations research and other aspects of applied mathematics. The theory of harmonic functions has been used to study and solve fluid flow problems (see \cite{aleman-2012}). The theory of univalent harmonic functions having prominent geometric properties like starlikeness, convexity and close-to-convexity appears naturally while dealing with planner fluid dynamical problems. For instance, the fluid flow problem on a convex domain satisfying an interesting geometric property has been extensively studied by Aleman and Constantin  \cite{aleman-2012}. With the help of geometric properties of harmonic mappings, Constantin and Martin \cite{constantin-2017} have obtained a complete solution of classifying all two dimensional fluid flows.\vspace{1.2mm}

The Bohr's phenomenon for the complex-valued harmonic mappings have been studied extensively by many authors (see \cite{Abu-CVEE-2010,Abu-JMAA-2014,Ahamed-CMFT-2022,Himadri-Vasu-P1,Ghosh-Allu-CVEE-2018,Nirupam-MonatsMath-2019,MBA-Sabir-MJM-2024}). For example, in $ 2016 $, Ali \emph{et al.} \cite{Ali & Abdul & Ng & CVEE & 2016} studied Bohr radius for the stralike log-harmonic mappings and obtain several interesting results.  Bohr-type inequalities for the class of harmonic $ p$-symmetric mappings and also for harmonic mappings with a multiple zero at the origin have been discussed by Huang \emph{et al.} \cite{Huang-Liu-Ponnusamy-MJM-2021}.  The Bohr radius for various classes of functions, for example, locally univalent harmonic mappings, $ K $-quasiconformal mappings, bounded harmonic functions, lacunary series have been studied extensively in \cite{MBA-Sabir-MJM-2024,Ismagilov- Kayumov- Ponnusamy-2020-JMAA,Kayumov-Ponnusamy-JMAA-2018,Kay & Pon & Sha & MN & 2018}. For more intriguing aspects of the Bohr phenomenon, we refer to the articles \cite{Ahamed-AMP-2021,Ahamed-CVEE-2021,Aizn-PAMS-2000,Alkhaleefah-Kayumov-Ponnusamy-PAMS-2019,Allu-JMAA-2021,Bhowmik-Das-JMAA-2018,Evdoridis-Ponnusamy-Rasila-IM-2018,Kay & Pon & AASFM & 2019,Kayumov-CRACAD-2020,Liu-Ponnusamy-BMMS-2019,Kayumov-Khammatova-Ponnusamy-JMAA-2021} and references therein.\vspace{1.5mm} 

In this section, with the help of sequence $\{ \varphi_n(r)\}_{n=0}^{\infty}$ of continuous functions in $r\in [0, 1)$, we study generalized version as well as improved versions of the Bohr inequality for certain classes of harmonic mappings on simple connected domain $ \Omega_{\gamma} $. A sense-preserving homeomorphism $f$ from the unit disk $\mathbb{D}$ onto $\Omega^{\prime}$ is said to be a $K$-quasiconformal mapping if, for $z\in\mathbb{D},$
\begin{align*}
	\dfrac{|f_z|+|f_{\overline{z}}|}{|f_z|-|f_{\overline{z}}|}=\dfrac{1+|\omega_f(z)|}{1-|\omega_f(z)|}\leq K,\;\mbox{i.e.},\; |\omega_f(z)|\leq k=\dfrac{K-1}{K+1},
\end{align*} 
where $K\geq 1$ so that $k\in [0,1)$ (see\cite{Duren-2004}). \vspace{1.5mm}

For the study of Bohr phenomenon for the class of $K$-quasiconformal harmonic mappings on $\mathbb{D}$ and several sharp Bohr inequalities, we refer to the articles \cite{Kay & Pon & Sha & MN & 2018,Ponnusamy-Vijayakumar,Ponnusamy-Vijayakumar-Wirths-JMAA-2022,Liu-Ponnusamy-BMMS-2019,Liu-Ponnusamy-Wang-RACSAM-2020,MBA-Sabir-MJM-2024}.\vspace{1.2mm}

We obtain the following result as a generalization of the Bohr inequality for $K$- quasiconformal sense-preserving harmonic mappings in $\mathbb{D}.$
\begin{thm}\label{BS-thm-3.111}
	Assume that $p\in (0,2]$ and let $f=h+\overline{g}=\sum_{n=0}^{\infty}a_nz^n+\overline{\sum_{n=1}^{\infty}b_nz^n}$  be a harmonic mapping in $\Omega_{\gamma}$ with $|h(z)|\leq 1$ on $\Omega_{\gamma}.$  If $h(z)=\sum_{n=0}^{\infty} a_nz^n$  and $g(z)=\sum_{n=0}^{\infty} b_nz^n$ in $\mathbb{D}$ and $|g^{\prime}(z)|\leq k|h^{\prime}(z)|$ for some $k\in [0,1),$ and $\{ \varphi_n(r)\}\in \mathcal{F}$ satisfies the inequality 
	\begin{align}\label{eqe-3.1}
		2(1+k)\sum_{n=1}^{\infty}\varphi_n(r)< p(\gamma+1)\varphi_0(r), 
	\end{align}
	then the following sharp inequality holds: 
	\begin{align*}
		\mathcal{N}_{ \varphi,f, p}(r):=|a_0|^p\varphi_0(r)+\sum_{n=1}^{\infty}\left(|a_n|+|b_n|\right)\varphi_n(r)\leq \varphi_0(r)\;\;\mbox{for}\;\; r\leq R^{\gamma,k}_{p,\varphi},
	\end{align*} 
	where $R^{\gamma,k}_{p,\varphi}$ is the minimal positive root in $(0,1)$ of the equation
	\begin{align}\label{eqe-3.2}
		2(1+k)\sum_{n=1}^{\infty}\varphi_n(r)= p(\gamma+1)\varphi_0(r). 
	\end{align}
	In the case when 
	\begin{align}\label{eqqq-3.3}
		2(1+k)\sum_{n=1}^{\infty}\varphi_n(r)> p(\gamma+1)\varphi_0(r), 
	\end{align}
	in some interval $(R^{\gamma,k}_{p,\varphi},R^{\gamma,k}_{p,\varphi}+\epsilon)$, the number $R^{\gamma,k}_{p,\varphi}$ cannot be improved.
	\end{thm}
	As a consequence of Theorem \ref{BS-thm-3.111}, in particular if $\varphi_n(r)=r^n$, then we obtain the following corollary which is exactly the result \cite[Theorem 4]{Evdoridis-Ponnusamy-Rasila-RM-2021}. 
	\begin{cor} (see \cite[Theorem 4]{Evdoridis-Ponnusamy-Rasila-RM-2021}\label{Cor-4.1})
		Let $f=h+\overline{g}=\sum_{n=0}^{\infty}a_nz^n+\overline{\sum_{n=1}^{\infty}b_nz^n}$  be a harmonic mapping in $\Omega_{\gamma}$ with $|h(z)|\leq 1$ on $\Omega_{\gamma}.$  If $h(z)=\sum_{n=0}^{\infty} a_nz^n$  and $g(z)=\sum_{n=0}^{\infty} b_nz^n$ in $\mathbb{D}$ and $|g^{\prime}(z)|\leq k|h^{\prime}(z)|$ for some $k\in [0,1),$ then
		\begin{align*}
			|a_0|+\sum_{n=1}^{\infty}\left(|a_n|+|b_n|\right)r^n\leq 1\;\;\mbox{for}\;\; r\leq R^k_{1, \gamma}=\dfrac{1+\gamma}{3+2k+\gamma}.
		\end{align*}  
		The radius $R^k_{1, \gamma}$ is the best possible.
	\end{cor}
	\begin{rem}
	We see that Corollary \ref{Cor-4.1} generalized the result \cite[Theorem 1.1]{Kay & Pon & Sha & MN & 2018} from unit disk $\mathbb{D}$ to simply connected domain $\Omega_{\gamma}$. In fact, Theorem \ref{BS-thm-3.111} is a two fold generalization of  \cite[Theorem 1.1]{Kay & Pon & Sha & MN & 2018} and \cite[Theorem 4]{Evdoridis-Ponnusamy-Rasila-RM-2021}.
	\end{rem}
It is well-known that the initial term in the Bohr inequality for any class of functions plays a crucial role in determining the Bohr radius. The Bohr inequality with the setting of $ |a_0|^2 $ for $ K $-quasiconformal harmonic mappings on $ \Omega_{\gamma} $ has not been established to date. In this regard, as a consequence of Theorem \ref{BS-thm-3.111}, we obtain the following result finding the precise Bohr radius.
	\begin{cor}\label{Cor-4.2} Assume the assumption of Corollary \ref{Cor-4.1}. Then
		\begin{align*}
			|a_0|^2+\sum_{n=1}^{\infty}\left(|a_n|+|b_n|\right)r^n\leq 1\;\;\mbox{for}\;\; r\leq R^k_{2, \gamma}=\dfrac{1+\gamma}{2+k+\gamma}.
		\end{align*}  
		The radius $R^k_{2, \gamma}$ is the best possible.
	\end{cor}
	\begin{rem}
		Corollary \ref{Cor-4.2} is a generalized version of the result \cite[Theore 1.2]{Kay & Pon & Sha & MN & 2018} from unit disk $\mathbb{D}$ to simply connected domain $\Omega_{\gamma}$. 
	\end{rem}
	Further for $k\to 1$ which corresponds to the limiting case $K\to\infty$ in Theorem \ref{BS-thm-3.111}, we obtain the following result for sense-preserving harmonic mappings.
	\begin{cor}
		Assume that $p\in (0,2]$ and let $f=h+\overline{g}=\sum_{n=0}^{\infty}a_nz^n+\overline{\sum_{n=1}^{\infty}b_nz^n}$  be a harmonic mapping in $\Omega_{\gamma}$ with $|h(z)|\leq 1$ on $\Omega_{\gamma}.$  If $h(z)=\sum_{n=0}^{\infty} a_nz^n$  and $g(z)=\sum_{n=0}^{\infty} b_nz^n$ in $\mathbb{D}$ and $f=h+\overline{g}$ is sense-preserving in $\mathbb{D}$ and $\{ \varphi_n(r)\}\in \mathcal{F}$ satisfies the inequality 
		\begin{align*}
		4\sum_{n=1}^{\infty}\varphi_n(r)< p(\gamma+1)\varphi_0(r), 
		\end{align*}
		then the following sharp inequality holds: 
		\begin{align*}
			|a_0|^p\varphi_0(r)+\sum_{n=1}^{\infty}\left(|a_n|+|b_n|\right)\varphi_n(r)\leq \varphi_0(r)\;\;\mbox{for}\;\; r\leq R^{\gamma,1}_{p,\varphi},
		\end{align*} 
		where $R^{\gamma,1}_{p,\varphi}$ is the minimal positive root in $(0,1)$ of the equation
		\begin{align*}
			4\sum_{n=1}^{\infty}\varphi_n(r)= p(\gamma+1)\varphi_0(r). 
		\end{align*}
		In the case when 
		\begin{align*}
			4\sum_{n=1}^{\infty}\varphi_n(r)> p(\gamma+1)\varphi_0(r), 
		\end{align*}
		in some interval $(R^{\gamma,1}_{p,\varphi},R^{\gamma,1}_{p,\varphi}+\epsilon)$, the number $R^{\gamma,1}_{p,\varphi}$ cannot be improved.
	\end{cor}
Our next aim is to find the Bohr inequality for $K$-quasiconformal, sense-preserving harmonic mappings, where the analytic part $h$ is subordinated to some analytic function $\psi$. Our purpose is to explore the cases in which $\psi$ is univalent and convex in $\mathbb{D}$. We present the following result.
\begin{thm}\label{BS-thm-3.2}
Suppose that $f(z)=h(z)+\overline{g(z)}=\sum_{n=0}^{\infty}a_nz^n+\overline{\sum_{n=1}^{\infty}b_nz^n}$ is a $K$- quasiconformal sense-preserving harmonic mapping in $\mathbb{D}$ and $h\prec \psi,$ where $\psi$ is a univalent and convex function in $\mathbb{D}.$ If $\{ \varphi_n(r)\}\in \mathcal{F}$  and satisfies the inequality 
\begin{align*}
2(1+k)\sum_{n=1}^{\infty}\varphi_n(r)< \varphi_0(r), 
\end{align*}
then the following sharp inequality holds:
\begin{align*}
\mathcal{Q}_{\varphi,f}(r):=\sum_{n=1}^{\infty}(|a_n|+|b_n|)\varphi_n(r)\leq \mathrm{dist}\left(\psi(0),\partial\psi(\mathbb{D})\right)\varphi_0(r)\;\;\mbox{for}\;\; |z|=r\leq R_{\varphi},
\end{align*}
where $ k\in [0,1)$ and $ R_{\varphi}$ is the minimal positive root in $(0,1)$ of the equation 
	\begin{align}\label{BS-eq-3.5}
		2(1+k)\sum_{n=1}^{\infty}\varphi_n(r)= \varphi_0(r).
	\end{align}
	In the case when 
	\begin{align*}
		2(1+k)\sum_{n=1}^{\infty}\varphi_n(r)> \varphi_0(r), 
	\end{align*}
	in some interval $(R_{\varphi},R_{\varphi}+\epsilon)$, the number $R_{\varphi}$ cannot be improved.
\end{thm}
As a consequence of Theorem \ref{BS-thm-3.2}, in particular, when  $\varphi_n(r)=r^n$, we obtain the following corollary. This shows that our result generalizes \cite[Theorem 1]{Liu-Ponnusamy-BMMS-2019}.
\begin{cor}(see \cite[Theorem 1]{Liu-Ponnusamy-BMMS-2019})
Suppose that $f(z)=h(z)+\overline{g(z)}=\sum_{n=0}^{\infty}a_nz^n+\overline{\sum_{n=1}^{\infty}b_nz^n}$ is a $K$-quasiconformal sense-preserving harmonic mapping in $\mathbb{D}$ and $h\prec \psi,$ where $\psi$ is a univalent and convex function in $\mathbb{D}.$ 
Then the following sharp inequality holds:
\begin{align*}
\sum_{n=1}^{\infty}(|a_n|+|b_n|)r^n\leq \mathrm{dist}\left(\psi(0),\partial\psi(\mathbb{D})\right)\;\;\mbox{for}\;\; |z|=r\leq \dfrac{K+1}{5K+1}.
\end{align*}
The constant $(K+1)/(5K+1)$ is sharp.
\end{cor}
The following lemmas will play key roles in proving the Theorems \ref{BS-thm-3.111} and \ref{BS-thm-3.2}.
	\begin{lem}(see \cite{Ponnusamy-Vijayakumar-Wirths-HJM-2021})\label{Lem-3.2}
	Let $\{\varphi_n(r)\}^{\infty}_{n=0}$ be a decreasing sequence of non negative functions in $[0,r_\varphi)$, and $g(z)=\sum_{n=0}^{\infty}b_nz^n,$ $h(z)=\sum_{n=0}^{\infty}a_nz^n$ analytic in $\mathbb{D}$ such that $|g^{\prime}(z)|\leq k |h^{\prime}(z)|$ in $\mathbb{D}$ and for some $k\in [0,1).$ Then 
	\begin{align*}
		\sum_{n=1}^{\infty}|b_n|^2\varphi_n(r)\leq k^2 \sum_{n=1}^{\infty}|a_n|^2\varphi_n(r)\;\;\mbox{ for}\;\; r\in [0,r_\varphi).
	\end{align*}
\end{lem}
\begin{lem}(see \cite{Duren-1983,Liu-Ponnusamy-BMMS-2019})\label{Lem-3.1}
	Let $\psi$ be an analytic univalent map from $\mathbb{D}$ onto a convex domain $\Omega=\psi(\mathbb{D}).$ Then
	\begin{align*}
		\dfrac{1}{2}|\psi^{\prime}(0)|\leq \mathrm{dist}\left(\psi(0), \partial\Omega\right)\leq |\psi^{\prime}(0)|.
	\end{align*}
	If $g(z)=\sum_{n=0}^{\infty}b_nz^n\prec \psi(z),$ then
	\begin{align*}
		|b_n|\leq |\psi^{\prime}(0)|\leq 2 \mathrm{dist}\left(\psi(0), \partial\Omega\right).
	\end{align*}
\end{lem}
\begin{proof}[\bf Proof of Theorem \ref{BS-thm-3.111}]
	By the assumption, we see that $h$ is analytic with $|h(z)|\leq 1$ on $\Omega_{\gamma}$. In view of Lemma \ref{Lem-3.2}, we obtain
	\begin{align*}
		\sum_{n=1}^{\infty}|b_n|^2\varphi_n(r)\leq k^2 \sum_{n=1}^{\infty}|a_n|^2\varphi_n(r).
	\end{align*}
By applying Lemma A and the well-known Schwarz inequality, a straightforward calculation yields that
	 \begin{align*}
	  	\sum_{n=1}^{\infty}|b_n|\varphi_n(r)&\leq \sqrt{\sum_{n=1}^{\infty}|b_n|^2\varphi_n(r)} \sqrt{\sum_{n=1}^{\infty}\varphi_n(r)}\\&\leq \sqrt{k^2 \sum_{n=1}^{\infty}|a_n|^2\varphi_n(r)}\sqrt{\sum_{n=1}^{\infty}\varphi_n(r)}\\&\leq \sqrt{k^2 \left(\dfrac{(1-|a_0|^2)}{1+\gamma}\right)^2\sum_{n=1}^{\infty}\varphi_n(r)}\sqrt{\sum_{n=1}^{\infty}\varphi_n(r)}\\&\leq k\dfrac{(1-|a_0|^2)}{1+\gamma}\Phi_1(r).
	  \end{align*} 
	 As a result, we obtain 
	  \begin{align*}
	 \mathcal{N}_{\varphi,f}(r)&\leq a^p\varphi_0(r)+(1+k)\dfrac{(1-a^2)}{1+\gamma}\Phi_1(r)\\&=\varphi_0(r)-(1-a^p)\varphi_0(r)+(1+k)\dfrac{(1-a^2)}{1+\gamma}\Phi_1(r)\\&=\varphi_0(r)+\frac{(1-a^2)}{(1+\gamma)}\bigg((1+k)\Phi_1(r)-\bigg(\frac{1-a^p}{1-a^2}\bigg)(1+\gamma)\varphi_0(r)\bigg).
	  \end{align*}
	 By considering the inequality \eqref{eee-2.5}, it can be seen that
	  \begin{align*}
	  	\mathcal{N}_{\varphi,f}(r)&\leq\varphi_0(r)+\frac{(1-a^2)}{(1+\gamma)}\bigg((1+k)\Phi_1(r)-\frac{p}{2}(1+\gamma)\varphi_0(r)\bigg)\\&=\varphi_0(r)+\frac{(1-a^2)}{2(1+\gamma)}\bigg(2(1+k)\Phi_1(r)-p(1+\gamma)\varphi_0(r)\bigg)\\&\leq \varphi_0(r),
	  \end{align*}
	  for $|z|=r\leq R^{\gamma,k}_{p,\varphi},$ where $R^{\gamma,k}_{p,\varphi}$ is the minimal positive root in $(0,1)$ of the equation given by \eqref{eqe-3.2}.\vspace{2mm}
	  
	  To show the sharpness part of the result, we consider the function $f_a(z)=h_a(z)+\overline{g_a(z)}$ in $\Omega_\gamma$ (see  \cite[Theorem 4]{Evdoridis-Ponnusamy-Rasila-RM-2021}), where 
	  $h_a$ is given by \eqref{eeee-2.6}, and $a_n$ $(n\geq 0)$ are given by \eqref{eee-2.7} and
	  \begin{align*}
	  	f_a(z)=k\lambda\left(h_a(z)-a_0\right).
	  \end{align*} 
	 Consequently, it becomes clear that
	  \begin{align*}
	  	 &\mathcal{N}_{\varphi,f_a}(r)\\&=\left(\dfrac{a-\gamma}{1-a\gamma}\right)^p\varphi_0(r)+(1+k\lambda)\sum_{n=1}^{\infty}\dfrac{1-a^2}{a(1-a\gamma)}\left(\dfrac{a(1-\gamma)}{1-a\gamma}\right)^n\varphi_n(r)\\&= \varphi_0(r)+\left(\bigg(\dfrac{a-\gamma}{1-a\gamma}\bigg)^p-1\right)\varphi_0(r)\\&\quad+(1+k\lambda)\dfrac{1-a^2}{a(1-a\gamma)}\sum_{n=1}^{\infty}\left(\dfrac{a(1-\gamma)}{1-a\gamma}\right)^n\varphi_n(r)\\&=\varphi_0(r)+\frac{(1-a)}{(1-\gamma)}\bigg(2(1+k)\Phi_1(r)-p(1+\gamma)\varphi_0(r)\bigg)\\&\quad+\bigg(p(1-a)\frac{1+\gamma}{1-\gamma}+\bigg(\frac{a-\gamma}{1-a\gamma}\bigg)^p-1\bigg)\varphi_0(r)\\&\quad+\frac{(1-a)}{(1-\gamma)}\bigg(-2(1+k)\Phi_1(r)+(1+k\lambda)\dfrac{(1+a)(1-\gamma)}{a(1-a\gamma)}\sum_{n=1}^{\infty}\left(\dfrac{a(1-\gamma)}{1-a\gamma}\right)^n\varphi_n(r)\bigg)\\&=\varphi_0(r)+\frac{(1-a)}{(1-\gamma)}\bigg(2(1+k)\Phi_1(r)-p(1+\gamma)\varphi_0(r)\bigg)+(1-a)A_{a,k,\gamma}(r),
	  \end{align*}
	 where
	 \begin{align*}
	 	&A_{a,k,\gamma}(r)\\&:=\frac{1}{(1-\gamma)}\bigg(-2(1+k)\Phi_1(r)+(1+k\lambda)\dfrac{(1+a)(1-\gamma)}{a(1-a\gamma)}\sum_{n=1}^{\infty}\left(\dfrac{a(1-\gamma)}{1-a\gamma}\right)^n\varphi_n(r)\bigg)\\&\quad+\bigg(\frac{p(1+\gamma)}{1-\gamma}+\frac{1}{1-a}\left(\bigg(\frac{a-\gamma}{1-a\gamma}\bigg)^p-1\right)\bigg)\varphi_0(r).
	 \end{align*}
	  Now taking $\lambda\to1^-$, it is a simple task to verify that 
	  \begin{align*}
	  	\lim\limits_{a\to1^-}&\bigg(-2(1+k)\Phi_1(r)+(1+k)\dfrac{(1+a)(1-\gamma)}{a(1-a\gamma)}\sum_{n=1}^{\infty}\left(\dfrac{a(1-\gamma)}{1-a\gamma}\right)^n\varphi_n(r)\bigg)=0
	  \end{align*}
	  and 
	  \begin{align*}
	  	&\lim\limits_{a\to1^-}\frac{1}{1-a}\left(\bigg(\frac{a-\gamma}{1-a\gamma}\bigg)^p-1\right)\;\;\; \left(\frac{0}{0}\; \mbox{form}\right)\\&=-\lim\limits_{a\to1^-}p\bigg(\frac{a-\gamma}{1-a\gamma}\bigg)^{p-1}\frac{1-\gamma^2}{(1-a\gamma)^2}\;\;\;\; \left(\mbox{by L'H$\hat{o}$pital's rule}\right)\\&=-\frac{p(1+\gamma)}{1-\gamma}.
	  \end{align*}
	  In view of the above estimates with  $\lambda\to1^-$, by a standard computation, we see that
	 \begin{align*}
	\lim\limits_{a\to1^-}A_{a,k,\gamma}(r)=\left(\frac{p(1+\gamma)}{1-\gamma}-\frac{p(1+\gamma)}{1-\gamma}\right)\varphi_0(r)=0.
	 \end{align*}
		Hence, it follows that 
	\begin{align*}
		\mathcal{N}_{\varphi,f_a}(r)=\varphi_0(r)+\frac{(1-a)}{(1-\gamma)}\bigg(2(1+k)\Phi_1(r)-p(1+\gamma)\varphi_0(r)\bigg)+O((1-a)^2)
	\end{align*}
	as $ a $ and $\lambda$ tend to $1^-$. Clearly, in view of \eqref{eeeq-2.4}, it is easy to see that 
$\mathcal{N}_{\varphi,f_a}(r)>\varphi_0(r)$
	when $a$ and $\lambda$ tend to $1^-$ and $r\in\left(R^{\gamma,k}_{p,\varphi},R^{\gamma,k}_{p,\varphi}+\epsilon \right)$, which shows that the radius $R^{\gamma,k}_{p,\varphi}$ cannot be improved further. This completes the proof.
\end{proof}

\begin{proof}[{\bf Proof of Theorem \ref{BS-thm-3.2}}]
	By the assumption $h\prec\psi$ and $\psi(\mathbb{D})$ is convex. 
	By the Lemma \ref{Lem-3.1}, we have $|a_n|\leq 2d,$ where  $d:=\left(\psi(0),\partial\psi(\mathbb{D})\right).$ Since $f=h+\overline{g}$ is a $K$- quasiconformal sense-preserving harmonic mapping in $\mathbb{D}$ so that $|g^{\prime}(z)|\leq k|h^{\prime}(z)|$ in $\mathbb{D},$ where $k\in [0,1)$, in view of Lemma \ref{Lem-3.2}, it follows that 
	\begin{align*}
		\sum_{n=1}^{\infty}|b_n|^2\varphi_n(r)\leq k^2 \sum_{n=1}^{\infty}|a_n|^2\varphi_n(r)\leq 4d^2k^2 \sum_{n=1}^{\infty}\varphi_n(r)=4d^2k^2 \Phi_1(r).
	\end{align*}
	Consequently, the classical Schwarz inequality leads to the conclusion that
	\begin{align*}
		\sum_{n=1}^{\infty}|b_n|\varphi_n(r)\leq \sqrt{\sum_{n=1}^{\infty}|b_n|^2\varphi_n(r)} \sqrt{\sum_{n=1}^{\infty}\varphi_n(r)}\leq \sqrt{4k^2d^2\Phi_1(r)}\sqrt{\Phi_1(r)}=2kd\Phi_1(r).
	\end{align*}
	Thus, the inequality
	\begin{align*}
		\mathcal{Q}_{\varphi,f}(r)\leq \sum_{n=1}^{\infty}2d \varphi_n(r)+k\sum_{n=1}^{\infty}2d \varphi_n(r)=2(1+k)d\sum_{n=1}^{\infty}\varphi_n(r)\leq d\varphi_0(r)
	\end{align*}
	holds for $|z|=r\leq R_{\varphi},$ where $R_{\varphi}$ is the minimal root in $(0,1)$ of the equation \eqref{BS-eq-3.5}. With this the desired inequality is established.\vspace{1mm}
	
	Next part of the proof is to show that the sharpness of the radius. Hence, we consider the function $f_a=h_a+\overline{g_a},$ where
	\begin{align*}
		\psi(z)=h_a(z)=\dfrac{1}{1-z}=\sum_{n=0}^{\infty}z^n
		\;\;\mbox{and}\;\; g_a^{\prime}(z)=k\lambda h_a^{\prime}(z), \;|\lambda|=1,\; z\in \mathbb{D}.
	\end{align*}
	It is well-known that $d=\mathrm{dist}\left(\psi(0),\partial\psi(\mathbb{D})\right)=1/2$ (see \cite{Liu-Ponnusamy-BMMS-2019}) and a simple computation shows that
	\begin{align*}
		g_a(z)=\lambda k \dfrac{z}{1-z}=k\lambda \sum_{n=1}^{\infty}z^n.
	\end{align*}
	For the function $f_a,$ we have 
	\begin{align*}
		\mathcal{Q}_{\varphi,f_a}(r)=\sum_{n=1}^{\infty}(1+|\lambda|k)\varphi_n(r)=(1+k)\sum_{n=1}^{\infty}\varphi_n(r)
	\end{align*}
	which is bigger than or equal to $\varphi_0(r)d=\varphi_0(r)/2$ if, and only if,  
	\begin{align*}
		2(1+k)\sum_{n=1}^{\infty}\varphi_n(r)> \varphi_0(r), 
	\end{align*}
	in some interval $(R_{\varphi},R_{\varphi}+\epsilon).$ This shows that the constant $R_{\varphi}$ is sharp. This completes the proof.
\end{proof}

\noindent{\bf Acknowledgment:} The authors would like to thank the referee(s) for the constructive suggestions and comments, which would be contributed to improving the presentation of the paper.

\section{Declarations}

\noindent\textbf{Conflict of interest:} The authors declare that there is no conflict  of interest regarding the publication of this paper.\vspace{1.5mm}

\noindent\textbf{Data availability statement:}  Data sharing not applicable to this article as no datasets were generated or analysed during the current study.\vspace{1.5mm}

\noindent{\bf Code availability:} Not Applicable.\vspace{1.5mm}

\noindent {\bf Authors' contributions:} All the authors have equal contributions.

\end{document}